\documentclass[12pt,twoside]{amsart}
\usepackage{amsmath, amsthm, amscd, amsfonts, amssymb, graphicx, color}
\usepackage[bookmarksnumbered, plainpages]{hyperref}

\newtheorem{thm}{Theorem}[section]

\newtheorem{lem}{Lemma}[section]
\newtheorem{prop}[thm]{Proposition}
\newtheorem{defn}[thm]{Definition}
\theoremstyle{remark}
\newtheorem{rmk}{Remark}[section]

\newtheorem{ex}{Example}[section]
\numberwithin{equation}{section}

\begin{document}

\vskip 1 true cm

\title{\bf Some Inequalities Related to Ricci Curvatures for Lagrangian Submanifolds of K$\ddot{\mathrm{A}}$hler QCH-manifolds}
\author{Liang Zhang$^{\ast}$, Xudong Liu, Dandan Cai}

\thanks{{\scriptsize
\newline $^{\ast}$Corresponding author\\
\hskip -0.4 true cm \textit{2010 Mathematics Subject Classification.} 53C40, 53C15.
}}

\maketitle

\begin{abstract}
 By establishing two general quadratic inequalities, we obtain some inequalities related to Ricci curvatures for Lagrangian submanifolds of K$\ddot{\mathrm{a}}$hler QCH-manifolds, which generalize some results for Lagrangian submanifolds of complex space forms.
\end{abstract}

\vskip 0.2 true cm


\pagestyle{myheadings}
\markboth{\rightline {\scriptsize L. Zhang, X. Liu and D. Cai}}
         {\leftline{\scriptsize Some Inequalities Related to Ricci Curvatures for Lagrangian Submanifolds of K$\ddot{\mathrm{a}}$hler QCH-manifolds}}

\bigskip
\bigskip


\section{ Introduction}
\vskip 0.4 true cm

One of the most important problems in submanifold theory is to find simple relationships between intrinsic and extrinsic invariants of a submanifold. The main extrinsic invariant is the squared mean curvature and the main intrinsic invariants include the Ricci curvature and the scalar curvature. In 1999, B.-Y.Chen \cite{BYC1} proved the following inequality on the Ricci curvature and the squared mean curvature $\|H\|^2$ for submanifolds of a real space form.

\begin{thm}[\cite{BYC1},Theorem 4]
  Let $\mathbb{R}^m(a)$ be a real space form of constant sectional curvature $a$, $N$ be an $n$-dimensional submanifold of $\mathbb{R}^m(a)$. Then the following statements are true:

  (i) For each unit vector $X\in T_pN$, we have
  \begin{equation}\label{1.1}
    \|H\|^2\geq \frac{4}{n^2}[Ric(X)-(n-1)a].
  \end{equation}

  (ii) If $H(p) = 0$, then a unit vector $X\in T_pM$ satisfies the equality case of (\ref{1.1})
   if and only if $X$ belongs to the relative null space $\mathcal{N}(p)$ given by
   \begin{equation*}
     \mathcal{N}(p)=\{X \in T_pN | h(X,Y)=0, \forall Y \in T_pN\}.
   \end{equation*}

   (iii) The equality case of (\ref{1.1}) holds for all unit vectors $X\in T_pN$ if and only if
   either $p$ is a geodesic point or $n=2$ and $p$ is an umbilical point.
\end{thm}

Inequality (\ref{1.1})is now named as the {\it Chen-Ricci inequality}. Afterwards, many papers studied similar Chen-Ricci inequalities for different kind of submanifolds in various ambient manifolds (cf.\cite{SPHMMT1,SPHMMT2,SPHMMT3,EKMMTMG1,JSKMKDMMT1,XL1,KMIMAO1,AMINR1,IM1,IM2,MMT1,MMTJSK1,PZLZ1,PZXLPLZ1}). Especially, B.-Y.Chen \cite{BYC2} proved (\ref{1.1}) also holds for Lagrangian submanifolds of a complex space form $M^n(a)$ of constant holomorphic sectional curvature $a$. But this inequality is not optimal in such a case. In fact it can be improved as follows.

\begin{thm}[\cite{SD1}, Theorem 3.1]\label{Theorem1.2}
  Let $N$ be a Lagrangian submanifold of real dimension $n(\geq 2)$ in a complex space form $M(a)$. Then for any point $p\in N$ and any unit vector $X\in T_pN$, we have
  \begin{equation}\label{1.2}
    Ric(X)\leq\frac{n-1}{4}a+\frac{(n-1)n}{4}\|H\|^2.
  \end{equation}
  The equality holds for all unit vectors in $T_pN$ if and only if either

  (i) $p$ is a totally geodesic point, or

  (ii) $n=2$ and $p$ is an $H$-umbilical point with $\lambda=3\mu$.
\end{thm}

\begin{rmk}
  By {\it a Lagrangian $H$-umbilical submanifold} of a K$\ddot{\mathrm{a}}$hler manifold we mean a Lagrangian submanifold whose second fundamental form takes the following simple form:
  \begin{align*}
    & h(e_1,e_1)=\lambda Je_1,\ h(e_2,e_2)=\cdots=h(e_n,e_n)=\mu Je_1,\\
    & h(e_1,e_j)=\mu Je_j,\ h(e_j,e_k)=0,\ j\not=k,\ j,k=2,\cdots,n
  \end{align*}
  for some suitable functions $\lambda$ and $\mu$ with respect to some suitable orthonormal local frame field. This concept was introduced by B.-Y.Chen in \cite{BYC3,BYC4} to find and investigate the "simplest" Lagrangian submanifolds next to the totally geodesic ones in complex space forms.
\end{rmk}

\begin{rmk}
  T.Oprea \cite{TO1} first proved the improved Chen-Ricci inequality (\ref{1.2}) by using an optimization technique. Afterwards, S.Deng provided another proof by establishing some elementary algebraic inequalities. In this way, he also completely characterized Lagrangian submanifolds satisfying the equality case.
\end{rmk}

In \cite{TO2} T.Oprea introduced another intrinsic invariant related to Ricci curvature for submanifolds as follows:

\begin{equation}\label{01.1}
  \delta_k:=\tau-\frac{1}{k-1}\min_{\substack{L,\dim L=k,\\ X\in L, \|X\|=1}} Ric_L(X),
\end{equation}
where $\tau$ is the scalar curvature of the submanifold, $L$ is a linear subspace of the tangent space of the submanifold, $Ric_L(X)$ is the Ricci curvature of the submanifold of $L$ at $X$.

For Lagrangian submanifolds of complex space forms, T.Oprea proved the following inequality for $\delta_n$.

\begin{thm}[\cite{TO2}, Theorem 4.2]\label{Theorem1.3}
  Let $N$ be a Lagrangian submanifold of real dimension $n(\geq 3)$ of a complex space form $M(a)$. Then
  \begin{equation*}
    \delta_n(M)\leq \frac{(n-2)(n+1)}{8}a+\frac{(3n-1)(n-2)n^2}{2(n-1)(3n+5)}\|H\|^2.
  \end{equation*}
\end{thm}

As a generalization of complex space forms, G.Ganchev and V.Mihova \cite{GGVM2,GGVM3} introduced the notion of {\it K$\ddot{a}$hler manifolds of quasi-constant holomorphic sectional curvatures} (briefly {\it K$\ddot{a}$hler QCH-manifolds}). This is the K$\ddot{\mathrm{a}}$hler analogue of the notion of a Riemannian manifold of quasi-constant sectional curvatures \cite{VBMP,GGVM1}. The main purpose of this paper is to provide some inequalities related to Ricci curvatures for Lagrangian submanifolds of this kind of ambient space, which can generalize Theorem \ref{Theorem1.2} and Theorem \ref{Theorem1.3}. In addition, we should point out that the method we used in this paper is different from T.Oprea's optimization technique, it can be viewed as a generalization of S.Deng's algebraic method. In fact, by using the theory of linear algebra, we can establish two general quadratic inequalities (see Theorem \ref{thm3.1} and Theorem \ref{thm3.2}), which can cover many special inequalities for the use of the proofs of Chen-Ricci inequalities and the inequalities related to the T.Oprea's invariant. And we think these general quadratic inequalities can also provide many other special inequalities which can be used  to obtain other kinds of geometric inequalities for submanifolds. For instance, we will use them to study inequalities related to Casorati curvatures for submanifolds in another paper.


\section{ Preliminaries}
\vskip 0.4 true cm

In this section, we recall some basic elements of the theory of K$\ddot{\mathrm{a}}$hler manifolds of quasi-constant holomorphic sectional curvatures and some basic formulas in the geometry of submanifolds.

Let $(M,g,J,D)$ be a 2n-dimensional K$\ddot{\mathrm{a}}$hler manifold with metric $g$, complex structure $J$ and $J$-invariant distribution $D$
of codimension 2. The Lie algebra of all $C^{\infty}$ vector fields on $M$ will be denoted by $\chi(M)$ and $T_pM$ will stand for the tangent space at any point $p\in M$. Assume that  $\xi$ is a local unit vector field around $p$ such that $D^{\bot}(p)=span\{ \xi,J\xi \}$, where $D^{\bot}(p)$ is the 2-dimensional $J$-invariant orthogonal complement to $D(p)$. Denote by $\eta$ and $\tilde{\eta}$ the unit 1-forms corresponding to $\xi$ and $J\xi$, respectively, i.e.,
\begin{equation}
  \eta(\bar X) = g(\xi, \bar X),\ \  \tilde{\eta}(\bar X) = g(J\xi, \bar X) = -\eta(J\bar X), \ \bar X\in\chi(M).
\end{equation}

Let $\bar\nabla$ be the Levi-Civita connection of the metric $g$. The Riemannian curvature tensor $\bar R$ of type (1,3), resp. (1,4), is
given by
\begin{equation*}
  \bar R(\bar X,\bar Y )\bar Z = \bar\nabla_{\bar X} \bar\nabla_{\bar Y} \bar Z -\bar\nabla_{\bar Y}\bar\nabla_{\bar X}\bar Z -\bar\nabla_{[\bar X,\bar Y ]}\bar Z,
\end{equation*}
\begin{equation*}
  \bar R(\bar X,\bar Y, \bar Z,\bar U) = g(\bar R(\bar X,\bar Y )\bar Z,\bar U),
\end{equation*}
for $\bar X,\bar Y,\bar Z,\bar U\in \chi(M)$.

There are three important invariant K$\ddot{\mathrm{a}}$hler tensors on $M$ defined as follows \cite{GGVM1}:

\begin{align}\label{QCH-pi}
  & 4\pi(\bar X,\bar Y,\bar Z,\bar U)\\\notag  = & g(\bar Y,\bar Z)g(\bar X,\bar U) -g(\bar X,\bar Z)g(\bar Y,\bar U)+g(J\bar Y,\bar Z)g(J\bar X,\bar U)-g(J\bar X,\bar Z)g(J\bar Y,\bar U)\\ \notag
   - & 2g(J\bar X,\bar Y )g(J\bar Z,\bar U),
\end{align}

\begin{align}\label{QCH-Phi}
  & 8\Phi(\bar X,\bar Y, \bar Z,\bar U)\\\notag  =& g(\bar Y, \bar Z)\{\eta(\bar X)\eta(\bar U) +\tilde\eta(\bar X)\tilde\eta(\bar U)\}-g(\bar X,\bar Z)\{\eta(\bar Y )\eta(\bar U) +\tilde\eta(\bar Y )\tilde\eta(\bar U)\}\\ \notag
  + &g(\bar X,\bar U)\{\eta(\bar Y )\eta(\bar Z) +\tilde\eta(\bar Y )\tilde\eta(\bar Z)\}-g(\bar Y,\bar U)\{\eta(\bar X)\eta(\bar Z) +\tilde\eta(\bar X)\tilde\eta(\bar Z)\}\\ \notag
   +&g(J\bar Y, \bar Z)\{\eta(\bar X)\tilde\eta(\bar U)-\eta(\bar U)\tilde\eta(\bar X)\}-g(J\bar X,\bar Z)\{\eta(\bar Y )\tilde\eta(\bar U) -\eta(\bar U )\tilde\eta(\bar Y)\}\\ \notag
   +&g(J\bar X,\bar U)\{\eta(\bar Y )\tilde\eta(\bar Z) -\eta(\bar Z )\tilde\eta(\bar Y)\}-g(J\bar Y,\bar U)\{\eta(\bar X)\tilde\eta(\bar Z) -\eta(\bar Z)\tilde\eta(\bar X)\}\\ \notag
   -&2g(J\bar X,\bar Y )\{\eta(\bar Z)\tilde\eta(\bar U)-\eta(\bar U)\tilde\eta(\bar Z)\}-2g(J\bar Z,\bar U)\{\eta(\bar X)\tilde\eta(\bar Y ) -\eta(\bar Y )\tilde\eta(\bar X)\},
\end{align}


\begin{align}\label{QCH-Psi}
  	& \Psi(\bar X, \bar Y ,\bar Z,\bar U)\\\notag = & \eta(\bar Y )\eta(\bar Z)\tilde\eta(\bar X)\tilde\eta(\bar U) - \eta(\bar X)\eta(\bar Z)\tilde\eta(\bar Y )\tilde\eta(\bar U)+\eta(\bar X )\eta(\bar U)\tilde\eta(\bar Y)\tilde\eta(\bar Z)\\\notag  - &\eta(\bar Y)\eta(\bar U)\tilde\eta(\bar X )\tilde\eta(\bar Z),\\ \notag
\end{align}
where $\bar X,\bar Y,\bar Z,\bar U\in \chi(M)$.

\begin{defn}(\cite{GGVM1})
 Let $(M, g, J, D)$ be a K$\ddot{\mathrm{a}}$hler manifold with ${\mathrm{dim}} M = 2n \geq 4$ and $J$-invariant distribution $D$ of codimension 2. The manifold
is said to be of {\it quasi-constant holomorphic sectional curvatures (a K$\ddot{\mathrm{a}}$hler QCH-manifold)} if for any holomorphic section
$span\{X, JX\}$ generated by the unit tangent vector $X\in T_pM, p \in M$ with $\varphi=\angle(span\{X, JX\}, span\{\xi, J\xi\})$ the
Riemannian sectional curvature\\ $R(X, JX, JX, X)$ may only depend on the point $p \in M$ and the angle $\varphi$, i.e.
\begin{equation*}
  \bar R(X, JX, JX, X)= f(p,\varphi),\ \  p \in M,  \varphi\in[0,\frac{\pi}{2}].
\end{equation*}
\end{defn}

\begin{rmk}
  This notion corresponds to the notion of a Riemannian manifold of quasi-constant sectional curvature \cite{VBMP,GGVM2}.
\end{rmk}

In \cite{GGVM1}, G.Ganchev and V.Mihova found a curvature identity characterizing K$\ddot{\mathrm{a}}$hler QCH-manifolds.

\begin{prop}\label{QCH-curvature}
  Let $(M, g, J, D)({\mathrm{dim}} M = 2n \geq 4)$ be a K\"{a}hler manifold with J invariant distribution $D$ of codimension 2. Then
$(M, g, J,D)$ is of quasi-constant holomorphic sectional curvatures if and only if
\begin{equation}\label{Riemanniancurvature}
  \bar R = a\pi + b\Phi + c\Psi,
\end{equation}
where $a, b$ and $c$ are functions on $M$ and the tensors $\pi,\Phi$ and $\Psi$  are given by (\ref{QCH-pi}), (\ref{QCH-Phi}) and (\ref{QCH-Psi}), respectively.
\end{prop}

Now suppose $N$ is a Lagrangian submanifold of $M$. This means the complex structure $J$ carries each tangent space of $N$ into its corresponding normal space. Denote by $R$ the Riemannian curvature tensor of $N$ associated to the induced Levi-Civita connection, $h$ the second fundamental form. Then the Gauss equation is

\begin{equation}\label{Gausseq}
  R(X,Y,Z,U)=\bar R(X,Y,Z,U)+g(h(X,U),h(Y,Z))-g(h(X,Z),h(Y,U)),
\end{equation}
where $X,Y,Z,U\in\chi(N)$.

At a point $p\in N$, we can choose an orthonormal basis
\begin{equation}\label{frame}
\{e_1,\cdots,e_n,Je_1,\cdots,Je_n\}
\end{equation}
of $T_pM$ such that $\{e_1,\cdots,e_n\}$ is a basis of $T_pN$. The mean curvature vector $H(p)$ is
\begin{equation*}
  H(p)=\frac{1}{n}\sum_{i=1}^n h(e_i,e_i),
\end{equation*}
thus
\begin{equation*}
  n^2\|H(p)\|^2=\sum_{i,j=1}^ng(h(e_i,e_i),h(e_j,e_j)).
\end{equation*}
The submanifold $N$ of $M$ is called {\it totally geodesic} if $h=0$, and it is called {\it minimal} if $H=0$. Write $h^r_{ij}=g(h(e_i,e_j),Je_r)$. Noting that $N$ is Lagrangian, it follows that \cite{BYCKO1}
\begin{equation}\label{symmetryofh}
  h^r_{ij}=h^i_{rj}=h^r_{ji},\ i,j,r=1,\cdots,n.
\end{equation}
The Gauss equation (\ref{Gausseq}) can be written in the form of components as follows
\begin{equation}\label{Gausseqcomponent}
  R(e_i,e_j,e_k,e_l)=\bar R(e_i,e_j,e_k,e_l)+\sum_{r=1}^n (h^r_{il}h^r_{jk}-h^r_{ik}h^r_{jl}).
\end{equation}

The scalar curvature $\tau(p)$ of $N$ at the point $p$ is defined by
\begin{equation}\label{scalarcuvature1}
  \tau(p)=\sum_{1\leq i<j\leq n}K(e_i\wedge e_j),
\end{equation}
where $K(e_i\wedge e_j)=R(e_i,e_j,e_j,e_i)$ is  the sectional curvature of $N$ of the plane section spanned by $e_i$ and $e_j$. From (\ref{Gausseqcomponent}) we have
\begin{equation}\label{scalarcuvature2}
  \tau(p)=\sum_{1\leq i<j\leq n}\bar R(e_i,e_j,e_j,e_i)+\sum_{r=1}^n\sum_{1\leq i<j\leq n}[h^r_{ii}h^r_{jj}-(h^r_{ij})^2].
\end{equation}

For any unit vector $X$ in $T_pN$, we may choose the orthonormal basis (\ref{frame}) such that $e_1=X$. Then the Ricci curvature of $N$ at $X$ is defined by
\begin{equation*}
  Ric(X)=\sum_{i=2}^n K(e_1\wedge e_i).
\end{equation*}
From (\ref{Gausseqcomponent}) we have
\begin{equation}\label{eqRic}
  Ric(X)=\sum_{i=2}^n\bar R(e_1,e_i,e_i,e_1)+\sum_{r=1}^n\sum_{i=2}^n[h^r_{11}h^r_{ii}-(h^r_{1i})^2].
\end{equation}


\section{Two general quadratic inequalities}
\vskip 0.4 true cm

In this section we will prove two general quadratic inequalities. These inequalities can cover many special quadratic inequalities (see, for example, \cite{SD1,TO2}) which are the key to prove the Chen-Ricci inequality and the inequality of T.Oprea's invariant. To prove the first one we need the following simple algebraic lemma.

\begin{lem}\label{lem3.1}
  If real numbers $a,b,c$ satisfy
  \begin{equation*}
    a+b+c\geq 0,\ ab+ac+bc\geq 0,\ abc\geq 0,
  \end{equation*}
  then $a,b,c$ are all non-negative.
\end{lem}

\begin{proof}
  Otherwise, there exists at least one negative among $a,b,c$. Without loss of generality, we may assume $c<0$. From $abc\geq 0$ we know that $ab\leq 0$, thus there exists one non-posibive between $a$ and $b$. Assume that $b\leq 0$. From $ab+ac+bc\geq 0$, we know that $bc\geq -a(b+c)$. From $a+b+c\geq 0$, we know that $a\geq -(b+c)>0$. Therefore, $bc\geq (b+c)^2$, which is equivalent to $(b+\frac{c}{2})^2+\frac{3}{4}c^2\leq 0$. It follows that $c=0$ which is a contradiction with $c<0$. Hence $a\geq 0, b\geq 0, c\geq 0$.
\end{proof}

\begin{thm}\label{thm3.1}
  Let $\mu,\alpha_1,\alpha_2,\beta,a$ be real numbers, $\alpha_1\not=\alpha_2$, $k_1$ and $k_2$ two non-negative integers, $k_1+k_2=n-1, n\geq 3$. Assume that $f(x_1,\cdots,x_n)$ is a quadratic form defined by
  \begin{equation*}
    f(x_1,\cdots,x_n)=\mu x_{1}^{2}+\alpha_{1}\sum\limits_{i=2}\limits^{k_{1}+1}x_{i}^{2}+\alpha_{2}\sum\limits_{i=k_{1}+2}\limits^{n}x_{i}^{2}+2a\sum\limits_{i=2}\limits^{n}x_{1}x_{i}+2\beta\sum\limits_{2\leq i<j\leq n}x_{i}x_{j}.
  \end{equation*}
  If $\mu,~\alpha_{1},~\alpha_{2},~\beta,a,~k_{1},~k_{2}$ satisfy the following conditions:

  \begin{align*}
    & \alpha_{1}\geq\beta,\alpha_{2}\geq\beta;\tag{\it A1}\label{A1}\\
    & \alpha_{1}+\alpha_{2}+(n-3)\beta+\mu\geq0;\tag{\it A2}\label{A2}\\
    & (\mu+\alpha_{1}-\beta)[\alpha_{2}+(k_{2}-1)\beta]+(\alpha_{1}-\beta)\mu+k_{1}\beta(\alpha_{2}-\beta+\mu)-(n-1)a^{2}\geq0;\tag{\it A3}\label{A3}\\
    & (\alpha_{1}-\beta)[\alpha_{2}+(k_{2}-1)\beta]\mu+k_{1}(\alpha_{2}-\beta)\beta\mu- a^{2}[k_{1}(\alpha_{2}-\beta)+k_{2}(\alpha_{1}-\beta)]\geq0;\tag{\it A4}\label{A4}\\
    & \mu+k_{1}\alpha_{1}+k_{2}\alpha_{2}>0;\tag{\it A5}\label{A5}\\
    & (\alpha_{1}-\beta)(\alpha_{2}-\beta)\{(\alpha_{1}-\beta)[\alpha_{2}+(k_{2}-1)\beta]\mu+k_{1}(\alpha_{2}-\beta)\beta\mu\tag{\it A6}\label{A6}\\
    & -a^{2}[k_{1}(\alpha_{2}-\beta)+k_{2}(\alpha_{1}-\beta)]\}=0,\\\notag
  \end{align*}
   then $f$ is positive semidefinite, i.e., $f(x_1,\cdots,x_n)\geq 0$. The equality case can be divided into the following cases:

  (B1) If $\alpha_{1}>\beta,~\alpha_{2}>\beta$, then $f(x_{1},...,x_{n})=0$ if and only if
  \begin{equation*}
    x_{2}=...=x_{k_{1}+1},~x_{k_{1}+2}=...=x_{n}=\frac{\alpha_{1}-\beta}{\alpha_{2}-\beta}x_{2},~\mu x_{1}=-a\sum\limits_{i=2}\limits^{n}x_{i};
  \end{equation*}

  (B2) If $\alpha_{1}>\beta,~\alpha_{2}=\beta,~\mu\beta=a^{2}$, then $f(x_{1},...,x_{n})=0$ if and only if
  \begin{equation*}
   x_{2}=...=x_{k_{1}+1}=0,~\mu x_{1}=-a\sum\limits_{i=k_{1}+2}\limits^{n}x_{i},~ax_1=-\beta\sum\limits_{i=k_{1}+2}\limits^{n}x_{i};
  \end{equation*}

  (B3) If $\alpha_{1}>\beta,~\alpha_{2}=\beta,~\mu\beta\not=a^{2}$, then $f(x_{1},...,x_{n})=0$ if and only if
  \begin{equation*}
   x_{1}=x_{2}=...=x_{k_{1}+1}=0,~\sum\limits_{i=k_{1}+2}\limits^{n}x_{i}=0;
  \end{equation*}

   (B4) If $\alpha_{1}=\beta,~\alpha_{2}>\beta,~\mu\beta=a^{2}$, then $f(x_{1},...,x_{n})=0$ if and only if
  \begin{equation*}
  x_{k_{1}+2}=...=x_{n}=0,~\mu x_{1}=-a\sum\limits_{i=2}\limits^{k_{1}+1}x_{i},~ax_1=-\beta\sum_{i=2}^{k_1+1}x_i;
  \end{equation*}

   (B5) If $\alpha_{1}=\beta,~\alpha_{2}>\beta,~\mu\beta\not=a^{2}$, then $f(x_{1},...,x_{n})=0$ if and only if
  \begin{equation*}
  x_1=0,~ x_{k_{1}+2}=...=x_{n}=0,~\sum\limits_{i=2}\limits^{k_{1}+1}x_{i}=0.
  \end{equation*}

%

  \end{thm}

  \begin{proof}
    The matrix of the quadratic form $f$ is
    \begin{equation*}
      C=\left(
       \begin{array}{ccccccc}
       \mu & a & \cdots & a & a & \cdots & a \\
       a & \alpha_{1} & \cdots & \beta & \beta & \cdots & \beta \\
       \vdots & \vdots & \ddots & \vdots & \vdots & \cdots & \vdots \\
        a & \beta & \cdots & \alpha_{1} & \beta & \cdots & \beta \\
         a & \beta & \cdots & \beta & \alpha_{2} & \cdots & \beta \\
         \vdots & \vdots & \cdots & \vdots & \vdots & \ddots & \vdots \\
         a & \beta & \cdots & \beta & \beta & \cdots & \alpha_{2} \\
       \end{array}\right).
    \end{equation*}
    By a direct calculation, the characteristic polynomial of $f$ is
    \begin{equation*}
      \det(\lambda I_n-C)=(\lambda-\alpha_{1}+\beta)^{k_{1}-1}(\lambda-\alpha_{2}+\beta)^{k_{2}-1}(\lambda^3-D_1\lambda^2+D_2\lambda-D_3),
    \end{equation*}
     where $I_n$ denotes the identity matrix, and
    \begin{align*}
      D_1 = & \alpha_{1}+\alpha_{2}+\mu+(n-3)\beta,\\
      D_2 = &  (\mu+\alpha_{1}-\beta)[\alpha_{2}+(k_{2}-1)\beta]+(\alpha_{1}-\beta)\mu+k_{1}\beta(\alpha_{2}-\beta+\mu)-(n-1)a^{2}, \\
      D_3 = & (\alpha_{1}-\beta)[\alpha_{2}+(k_{2}-1)\beta]\mu-k_{1}(\alpha_{2}-\beta)\beta\mu+a^{2}[k_{1}(\alpha_{2}-\beta)+k_{2}(\alpha_{1}-\beta)].
   \end{align*}
   Denote by $\lambda_1,\cdots,\lambda_n$ the eigenvalues of $f$, then
   \begin{equation*}
     \lambda_{1}=\cdots=\lambda_{k_{1}-1}=\alpha_{1}-\beta,\ \lambda_{k_{1}}=\cdots=\lambda_{n-3}=\alpha_{2}-\beta,
   \end{equation*}
   $\lambda_{n-2},\lambda_{n-1},\lambda_{n}$ are the roots of the equation
   \begin{equation*}
     \lambda^3-D_1\lambda^2+D_2\lambda-D_3=0.
   \end{equation*}
  It is clear that the condition \eqref{A1} guarantees that $\lambda_1,\cdots,\lambda_{n-3}$ are non-negative. According to Lemma\eqref{lem3.1} and the relationship between the roots and the coefficients of cubic equations, we see that condition \eqref{A2},\eqref{A3} and \eqref{A4} guarantee that $\lambda_{n-2},\lambda_{n-1},\lambda_{n}$ are non-negative. Noting that
  \begin{equation*}
    \sum_{i=1}^{n}\lambda_i={\rm trace}\ C=\mu+k_1\alpha_1+k_2\alpha_2,
  \end{equation*}
  thus condition \eqref{A5} implies that there exists at least one positive among $\lambda_1,\cdots,\lambda_n$. Similarly, from
  \begin{equation*}
    \lambda_1\cdots\lambda_n  =\det C= (\alpha_1-\beta)^{k_1-1}(\alpha_2-\beta)^{k_2-1}D_3,
  \end{equation*}
  and condition \eqref{A6}, we see that there exists at least one $0$ among $\lambda_1,\cdots,\lambda_n$. Therefore, $f$ is positive semidefinite, i.e., $f(x_1,\cdots,x_n)\geq 0$.

  For the equality case, noting that $f(x_1,\cdots,x_n)=0$ if and only if $(x_1,\cdots,x_n)$ is the eigenvector corresponding to the zero eigenvalue, so we should find the solutions to the homogeneous equation $CX= 0$, which is equivalent to
  \begin{equation}\label{eq3.1}
  \left\{
    \begin{aligned}
     & \mu x_{1}+a\sum\limits_{i=2}\limits^{n}x_{i}=0,\\
     &ax_{1}+\beta \sum\limits_{i=2}\limits^{n} x_{i}+(\alpha_{1}-\beta) x_{2}=0,\\
     &\ \ \ \cdots\  \cdots\  \cdots\\\
     & ax_{1}+\beta \sum\limits_{i=2}\limits^{n} x_{i}+(\alpha_{1}-\beta) x_{k_1+1}=0,\\
     & ax_{1}+\beta \sum\limits_{i=2}\limits^{n} x_{i}+(\alpha_{2}-\beta) x_{k_1+2}=0,\\
     & \ \ \ \cdots \cdots\ \cdots\\
     & ax_{1}+\beta \sum\limits_{i=2}\limits^{n} x_{i}+(\alpha_{2}-\beta) x_{n}=0,
     \end{aligned}
      \right.
  \end{equation}
  If $\alpha_1>\beta,~\alpha_2>\beta$, then the solutions to \eqref{eq3.1} is
  \begin{equation*}
    x_{2}=...=x_{k_{1}+1},~x_{k_{1}+2}=...=x_{n}=\frac{\alpha_{1}-\beta}{\alpha_{2}-\beta}x_{2},\ \mu x_1=-a\sum_{i=2}^n x_i.
  \end{equation*}
  If $\alpha_{1}>\beta,~\alpha_{2}=\beta$, then the solutions to \eqref{eq3.1} satisfy
  \begin{equation*}
    x_2=\cdots=x_{k_1+1}=0.
  \end{equation*}
  Thus \eqref{eq3.1} becomes
  \begin{equation*}
    \left\{
     \begin{aligned}
       \mu x_{1}+a\sum\limits_{i=k_1+2}\limits^{n}x_{i}&=0,\\ax_{1}+\beta \sum\limits_{i=k_{1}+2}\limits^{n} x_{i}&=0.
     \end{aligned}
     \right.\
  \end{equation*}
  If $\mu\beta\not=a^2$, then $x_1=0,\ \sum\limits_{i=k_1+2}^nx_i=0$; If $\mu\beta=a^2$, then $\mu x_{1}=-a\sum\limits_{i=k_{1}+2}\limits^{n}x_{i}$, $ax_1=-\beta\sum\limits_{i=k_{1}+2}\limits^{n}x_{i}$. This proves (B2) and (B3). Similarly, one can prove (B4) and (B5) in the same way.

  \end{proof}

  Similarly, we can obtain the second general quadratic inequality as follows.

  \begin{thm}\label{thm3.2}
  Let $\mu,\alpha,\beta,a$ be real numbers, $n\geq 2$ a positive integer. Assume that $f(x_1,\cdots,x_n)$ is a quadratic form defined by
  \begin{equation*}
    f(x_1,\cdots,x_n)=\mu x_{1}^{2}+\alpha\sum\limits_{i=2}\limits^{n}x_{i}^{2}+2a\sum\limits_{i=2}\limits^{n}x_{1}x_{i}+2\beta\sum\limits_{2\leq i<j\leq n}x_{i}x_{j}.
  \end{equation*}
  If $\mu,~\alpha,~\beta,~a$ satisfy the following conditions:

  \begin{align*}
    & \alpha\geq\beta;\tag{\it A1}\label{A1'}\\
    & \mu +\alpha+(n-2)\beta\geq0;\tag{\it A2}\label{A2'}\\
    & \mu\alpha+(n-2)\mu\beta-(n-1)a^2\geq0;\tag{\it A3}\label{A3'}\\
    & \mu+(n-1)\alpha>0;\tag{\it A4}\label{A4'}\\
    & (\alpha-\beta)[\mu\alpha+(n-2)\mu\beta-(n-1)a^2]=0,\ \ \ \ \ \ \ \ \ \ \ \ \ \ \ \ \ \ \ \ \ \ \ \ \ \ \ \ \ \ \ \ \ \ \ \ \ \ \ \ \ \ \ \ \ \ \tag{\it A5}\label{A5'}\\\notag
  \end{align*}
   then $f$ is positive semidefinite, i.e., $f(x_1,\cdots,x_n)\geq 0$. The equality case can be divided into the following cases:

  (B1) If $\alpha>\beta$, then $f(x_{1},...,x_{n})=0$ if and only if
  \begin{equation*}
    x_{2}=...=x_{n},~\mu x_1=-(n-1)ax_2;
  \end{equation*}

  (B2) If $\alpha=\beta,~\mu\alpha\not=a^{2}$, then $f(x_{1},...,x_{n})=0$ if and only if
  \begin{equation*}
   x_{1}=0,~\sum_{i=2}^nx_i=0;
  \end{equation*}

  (B3) If $\alpha=\beta,~\mu\alpha=a^{2}$, then $f(x_{1},...,x_{n})=0$ if and only if
  \begin{equation*}
  \mu x_1=-a\sum_{i=2}^n x_i.
  \end{equation*}

  \end{thm}

  \begin{rmk}
    The proof of Theorem \ref{thm3.2} is the same as Theorem \ref{thm3.1}, one should only note that the characteristic polynomial of $f$ in Theorem \ref{thm3.2} is
    \begin{equation*}
      (\lambda-\alpha+\beta)^{n-2}\{\lambda^2-[\mu+\alpha+(n-2)\beta]\lambda+\mu\alpha+(n-2)\mu\beta-(n-1)a^2\}.
    \end{equation*}
    So we omit the proof here.
  \end{rmk}

  \begin{ex}[Cauchy-Schwartz]\label{ex3.1}
    Let $x_1,\cdots, x_n\in \mathbb{R}$,\ $n\geq2$, then
    \begin{equation}\label{3.1}
     \sum_{i=1}^{n}x_{i}^2\geq \frac{1}{n}(\sum_{i=1}^{n}x_i)^2.
    \end{equation}
    The equality holds if and only if
    \begin{equation*}
      x_1=\cdots=x_n.
    \end{equation*}
  \end{ex}

   \begin{proof}
    By setting
    \begin{equation*}
      \mu=\alpha=1-\frac{1}{n},~ \beta= a=-\frac{1}{n}
    \end{equation*}
    in Theorem \ref{thm3.2}, one can verify that they satisfy the conditions \eqref{A1'}-\eqref{A5'}. Thus
    \begin{equation*}
     (1-\frac{1}{n})\sum_{i=1}^{n}x_{i}^2-\frac{2}{n}\sum_{1\leq i<j\leq n}x_ix_j\geq0,
    \end{equation*}
    which is equivalent to \eqref{3.1}.
    The above $\mu,~\alpha,~\beta,~a$ also satisfy the condition of the case (B1) in Theorem \ref{thm3.2}. Hence the equality holds if and only if $x_1=\cdots=x_n$.
  \end{proof}

  \begin{ex}[\cite{SD1},~Lemma 2.2]\label{ex3.2}
    Let $x_1,\cdots, x_n\in \mathbb{R}$,\ $n\geq 2$, then
    \begin{equation}\label{3.2}
      \sum_{i=2}^{n}x_{1}x_{i}-\sum_{i=2}^{n}x_{i}^{2}\leq\frac{n-1}{4n}(\sum_{i=1}^{n}x_{i})^{2}.
    \end{equation}
    The equality holds if and only if
    \begin{equation*}
      x_2=\cdots=x_n=\frac{1}{n+1}x_1.
    \end{equation*}
  \end{ex}

  \begin{proof}
    By setting
    \begin{equation*}
      \mu=\frac{n-1}{4n},~ \alpha=\frac{5n-1}{4n},~ \beta=\frac{n-1}{4n},~ a=-\frac{n+1}{4n}
    \end{equation*}
    in Theorem \ref{thm3.2}, one can verify that they satisfy the conditions \eqref{A1'}-\eqref{A5'}. Thus
    \begin{equation*}
      \frac{n-1}{4n}x_{1}^{2}+\frac{5n-1}{4n}\sum_{i=2}^{n}x_{i}^{2}-\frac{n+1}{2n}\sum_{i=2}^{n}x_{1}x_{i}+\frac{n-1}{2n}\sum_{2\leq i<j\leq n}x_{i}x_{j}\geq0,
    \end{equation*}
    which is equivalent to \eqref{3.2}.
    The above $\mu,~\alpha,~\beta,~a$ also satisfy the condition of the case (B1) in Theorem \ref{thm3.2}. Hence the equality holds if and only if $ x_2=\cdots=x_n=\frac{1}{n+1}x_1$.
  \end{proof}

 \begin{ex}[\cite{SD1}, Lemma 3.3]\label{ex3.3}
    Let $x_1,\cdots, x_n\in \mathbb{R}$,\ $n\geq 2$, then
    \begin{equation}\label{3.3}
      \sum_{i=2}^{n}x_{1}x_{i}-x_{1}^{2}\leq\frac{1}{8}(\sum_{i=1}^{n}x_{i})^{2}.
    \end{equation}
    The equality holds if and only if
    \begin{equation*}
      x_{2}+\cdots+x_{n}=3x_{1}.
    \end{equation*}
  \end{ex}

  \begin{proof}
    By setting
    \begin{equation*}
      \mu=\frac{9}{8},~\alpha=\beta=\frac{1}{8},~ a=-\frac{3}{8}
    \end{equation*}
    in Theorem \ref{thm3.2}, one can verify that they satisfy the conditions \eqref{A1'}-\eqref{A5'}. Thus
    \begin{equation*}
      \frac{9}{8}x_{1}^{2}+\frac{1}{8}\sum_{i=2}^{n}x_{i}^{2}-\frac{3}{4}\sum_{i=2}^{n}x_{1}x_{i}+\frac{1}{4}\sum_{2\leq i<j\leq n}x_{i}x_{j}\geq0,
    \end{equation*}
    which is equivalent to \eqref{3.3}
    The above $\mu,~\alpha,~\beta,~a$ also satisfy the condition of the case (B3) in Theorem \ref{thm3.2}. Hence the equality holds if and only if $ x_{2}+\cdots+x_{n}=3x_{1}$.
  \end{proof}

  \begin{ex}[\cite{TO2}, (4.14)]\label{ex3.4}
    Let $x_1,\cdots, x_n\in \mathbb{R}$,\ $n\geq3$, then
    \begin{equation}\label{3.4}
      -(n-2)\sum_{i=2}^{n}x_{i}^{2}+(n-2)\sum_{i=2}^{n}x_{1}x_{i}+(n-1)\sum_{2\leq i<j\leq n}x_{i}x_{j}\leq\frac{(n-2)(n-1)}{2(n+1)}(\sum_{i=1}^{n}x_{i})^{2}.
    \end{equation}
    The equality holds if and only if
    \begin{equation*}
      x_{2}=\cdots=x_{n}=\frac{1}{2}x_{1}.
    \end{equation*}
  \end{ex}

  \begin{proof}
    By setting
    \begin{equation*}
      \mu=\frac{(n-2)(n-1)}{2(n+1)},~\alpha=\frac{(3n+1)(n-2)}{2(n+1)},~ \beta=-\frac{3(n-1)}{2(n+1)},~a=-\frac{n-2}{n+1}
    \end{equation*}
    in Theorem \ref{thm3.2}, one can verify that they satisfy the conditions \eqref{A1'}-\eqref{A5'}. Thus
    \begin{equation*}
      \frac{(n-2)(n-1)}{2(n+1)}x_{1}^{2}+\frac{(3n+1)(n-2)}{2(n+1)}\sum_{i=2}^{n}x_{i}^{2}-\frac{2(n-2)}{n+1}\sum_{i=2}^{n}x_{1}x_{i}-\frac{3(n-1)}{n+1}\sum_{2\leq i<j\leq n}x_{i}x_{j}\geq0,
    \end{equation*}
    which is equivalent to \eqref{3.4}.
    The above $\mu,~\alpha,~\beta,~a$ also satisfy the condition of the case (B1) in Theorem \ref{thm3.2}. Hence the equality holds if and only if $x_{2}=\cdots=x_{n}=\frac{1}{2}x_{1}$.
  \end{proof}

  \begin{ex}[\cite{TO2}, (4.27)]\label{ex3.5}
    Let $y_1,\cdots, y_n\in \mathbb{R}$,\ $n\geq3$, then for any positive integer $r\in [2,n]$, we have
    \begin{equation*}
      -(n-2)y_{1}^{2}-(n-1)\sum_{\substack{1\leq i\leq n\\ i(\not=1,r)}}y_{i}^{2}+(n-2)\sum_{i=2}^{n}y_{1}y_{i}
    +(n-1)\sum_{2\leq i< j\leq n}y_{i}y_{j}\leq\frac{(3n-1)(n-2)}{2(3n+5)}(\sum_{i=1}^{n}y_{i})^{2}.
    \end{equation*}
    The equality holds if and only if
    \begin{equation*}
      y_{2}=\cdots=y_{r-1}=y_{r+1}=\cdots=y_{n}=\frac{3}{2}y_{1},~y_{r}=\frac{9}{2}y_{1}..
    \end{equation*}
  \end{ex}

  \begin{proof}
    By setting
    \begin{equation*}
      \mu=\frac{9(n-2)(n+1)}{2(3n+5)},~~\alpha_{1}=\frac{(3n-1)(n-2)}{2(3n+5)},~\alpha_{2}=\frac{9n^{2}-3n-8}{2(3n+5)},
    \end{equation*}
    \begin{equation*}
     \beta=-\frac{9n-7}{2(3n+5)}, ~ a=-\frac{3(n-2)}{3n+5},~ k_{1}=1,~ k_{2}=n-2,
    \end{equation*}
    in Theorem \ref{thm3.1}, one can verify that they satisfy the conditions \eqref{A1}-\eqref{A6}. Thus
    \begin{align*}
      & \frac{9(n-2)(n+1)}{2(3n+5)}x_{1}^{2}+\frac{(3n-1)(n-2)}{2(3n+5)}x_{2}^{2}+\frac{9n^{2}-3n-8}{2(3n+5)}\sum_{i=3}^{n}x_{i}^{2}\\
      & -\frac{6(n-2)}{3n+5}\sum_{i=2}^{n}x_{1}x_{i}-\frac{9n-7}{3n+5}\sum_{2\leq i<j\leq n}x_{i}x_{j}\\
      & \geq 0,
    \end{align*}
    which is equivalent to
    \begin{align}\label{eq3.2}
      & -(n-2)x_{1}^{2}-(n-1)\sum_{i=3}^{n}x_{i}^{2}+(n-2)\sum_{i=2}^{n}x_{1}x_{i}
       +(n-1)\sum_{2\leq i< j\leq n}x_{i}x_{j}\\\notag
      & \leq  \frac{(3n-1)(n-2)}{2(3n+5)}(\sum_{i=1}^{n}x_{i})^{2}.
    \end{align}
    The above $\mu,\alpha_1,\alpha_2,\beta,a$ also satisfy the condition of the case (B1) in Theorem \ref{thm3.1}. Hence the equality holds if and only if
    \begin{equation*}
      x_{3}=\cdots=x_{n}=\frac{3}{2}x_{1},~x_{2}=\frac{9}{2}x_{1}.
    \end{equation*}

    For fixed positive integer $r\in\{2,\ldots,n\}$, if we set
    \begin{equation*}
      y_{2}=x_{r},~y_{r}=x_{2},~y_{i}=x_{i}~ i\neq2,r,
    \end{equation*}
    then the conclusion is immediately followed.
  \end{proof}


\section{ Chen-Ricci inequality}
\vskip 0.4 true cm

In this section, we prove the Chen-Ricci inequality for Lagrangian submanifolds of K$\ddot{\mathrm{a}}$hler QCH-manifolds.

\begin{thm}\label{thm4.1}
  Let $N$ be a Lagrangian submanifold of real dimension $n(\geq 2)$ of a K$\ddot{\mathrm{a}}$hler QCH-manifolds $M$. Then for any point $p\in N$ and any unit vector $X\in T_pN$, we have
  \begin{align*}
    Ric(X) & \leq\frac{n-1}{4}a+\frac{1}{8}\{(n-2)[\eta(X)^2+\tilde\eta(X)^2]+1\}b\\
    & +\|\eta(X)\tilde\eta^{\top}-\tilde\eta(X)\eta^{\top}\|^2c+\frac{(n-1)n}{4}\|H\|^2,
  \end{align*}
  where $\eta^{\top}$ and $\tilde\eta^{\top}$ are the tangential components of $\eta$ and $\tilde\eta$, respectively. The equality case holds for  all unit vectors in $T_pN$ if and only if either

  (i) $p$ is a totally geodesic point or

  (ii) $n=2$ and $p$ is a $H$-umbilical point with $\lambda=3\mu$.
\end{thm}

\begin{rmk}
  For $b=0,c=0$, Theorem \ref{thm4.1} is due to Theorem \ref{Theorem1.2}.
\end{rmk}

\begin{proof}
  Let $X$ be a unit vector in $T_pN$. We choose an orthonormal basis $\{e_1,\cdots,e_n\}$ in $T_pN$ such that $e_1=X$. First we calculate the Ricci curvature of $N$ according to \eqref{eqRic}.
  It follows from \eqref{Riemanniancurvature} that
  \begin{align}\label{eq4.2}
    \sum_{i=2}^n\bar R(e_1,e_i,e_i,e_1) & = a\sum_{i=2}^n\pi(e_1,e_i,e_i,e_1)+b\sum_{i=2}^n\Phi(e_1,e_i,e_i,e_1)\\\notag
    & +c\sum_{i=2}^n\Psi(e_1,e_i,e_i,e_1).
  \end{align}
  According to \eqref{QCH-pi}, \eqref{QCH-Phi}, \eqref{QCH-Psi}, and noting that $g(e_i,e_j)=\delta_{ij}$, $g(e_i,Je_j)=0$ for $i,j=1,\cdots,n$, by direct calculations, we have
  \begin{equation}\label{eq4.3}
    \sum_{i=2}^n\pi(e_1,e_i,e_i,e_1)=\frac{n-1}{4},
  \end{equation}
  \begin{align}\label{eq4.4}
    8\sum_{i=2}^n\Phi(e_1,e_i,e_i,e_1)& =(n-1)[\eta(e_1)^2+\tilde\eta(e_1)^2]+\sum_{i=2}^n[\eta(e_i)^2+\tilde\eta(e_i)^2]\\\notag
    & =(n-2)[\eta(e_1)^2+\tilde\eta(e_1)^2]+\sum_{i=1}^n[\eta(e_i)^2+\eta(Je_i)^2]\\\notag
    & =(n-2)[\eta(X)^2+\tilde\eta(X)^2]+1,
  \end{align}
  \begin{align}\label{eq4.5}
    & \sum_{i=2}^n\Psi(e_1,e_i,e_i,e_1)\\\notag
    =& \eta(e_1)^2\sum_{i=2}^n\tilde\eta(e_i)^2+\tilde\eta(e_1)^2\sum_{i=2}^n\eta(e_i)^2-2\eta(e_1)\tilde\eta(e_1)\sum_{i=2}^n\eta(e_i)\tilde\eta(e_i)\\\notag
     = &\eta(X)^2\|\tilde\eta^{\top}\|^2+\tilde\eta(X)^2\|\eta^{\top}\|^2-2\eta(X)\tilde\eta(X)g(\eta^{\top},\tilde\eta^{\top})\\\notag
     =& \|\eta(X)\tilde\eta^{\top}-\tilde\eta(X)\eta^{\top}\|^2.
  \end{align}
  By substituting \eqref{eq4.3}, \eqref{eq4.4} and \eqref{eq4.5} into \eqref{eq4.2}, we get
  \begin{align}\label{eq4.6}
    \sum_{i=2}^n\bar R(e_1,e_i,e_i,e_1) & =\frac{n-1}{4}a+\frac{1}{8}\{(n-2)[\eta(X)^2+\tilde\eta(X)^2]+1\}b\\\notag
     & +\|\eta(X)\tilde\eta^{\top}-\tilde\eta(X)\eta^{\top}\|^2c.
  \end{align}

  On the other hand, by using \eqref{symmetryofh} we have

  \begin{align}\label{eq4.7}
    \sum_{r=1}^n\sum_{i=2}^n[h^r_{11}h^r_{ii}-(h^r_{1i})^2] & =\sum_{r=1}^n\sum_{i=2}^nh^r_{11}h^r_{ii}-\sum_{i=2}^n(h^i_{11})^2-\sum_{r,i=2}^n(h^1_{ri})^2\\\notag
    & \leq \sum_{r=1}^n\sum_{i=2}^nh^r_{11}h^r_{ii}-\sum_{r=2}^n(h^r_{11})^2-\sum_{i=2}^n(h^1_{ii})^2\\\notag
    & :=f_1(h^1_{11},\cdots,h^1_{nn})+\sum_{r=2}^n f_r(h^r_{11},\cdots,h^r_{nn}),
  \end{align}
  where $f_1,f_r:\mathbb{R}^n\to\mathbb{R},~r=2,\cdots,n$ are quadratic forms defined respectively by
  \begin{align*}
    f_1(h^1_{11},\cdots,h^1_{nn}) & =h^1_{11}\sum_{i=2}^nh^1_{ii}-\sum_{i=2}^n(h^1_{ii})^2,\\
    f_r(h^r_{11},\cdots,h^r_{nn}) & =h^r_{11}\sum_{i=2}^nh^r_{ii}-(h^r_{11})^2.
  \end{align*}
  From Example \ref{ex3.2}, we know that
  \begin{equation}\label{eq4.8}
    f_1(h^1_{11},\cdots,h^1_{nn})\leq\frac{n-1}{4n}(\sum_{i=1}^nh^1_{ii})^2,
  \end{equation}
  with the equality holding if and only if
  \begin{equation}\label{eq4.9}
   \frac{1}{n+1}h^1_{11}=h^1_{22}=\cdots=h^1_{nn}.
  \end{equation}
  From Example \ref{ex3.3}, we know that
  \begin{equation}\label{eq4.10}
    f_r(h^r_{11},\cdots,h^r_{nn})\leq\frac{1}{8}(\sum_{i=1}^nh^r_{ii})^2,~r=2,\cdots,n
  \end{equation}
  with the equality holding if and only if
  \begin{equation}\label{eq4.11}
   3h^r_{11}=h^r_{22}+\cdots+h^r_{nn}.
  \end{equation}
  From \eqref{eqRic}, \eqref{eq4.6}, \eqref{eq4.7}, \eqref{eq4.8}, \eqref{eq4.10}, and noting that $n\geq 2$, we have
  \begin{align}\label{eq4.12}
   Ric(X) & \leq\frac{n-1}{4}a+\frac{1}{8}\{(n-2)[\eta(X)^2+\tilde\eta(X)^2]+1\}b+\|\eta(X)\tilde\eta^{\top}-\tilde\eta(X)\eta^{\top}\|^2c\\\notag
    & +\frac{n-1}{4n}(\sum_{i=1}^nh^1_{ii})^2+\frac{1}{8}\sum_{r=2}^n(\sum_{i=1}^nh^r_{ii})^2\\\notag
    & \leq\frac{n-1}{4}a+\frac{1}{8}\{(n-2)[\eta(X)^2+\tilde\eta(X)^2]+1\}b+\|\eta(X)\tilde\eta^{\top}-\tilde\eta(X)\eta^{\top}\|^2c\\\notag
    & +\frac{(n-1)n}{4}\|H\|^2.\notag
  \end{align}

  Now we consider the equality case. Suppose the equality of \eqref{eq4.12} holds for any unit vector $X$ in $T_pN$. For $n\geq 3$, it follows from \eqref{eq4.12} that
  \begin{equation}\label{eq4.13}
    \sum_{i=1}^nh^r_{ii}=0,~r=2,\cdots,n.
  \end{equation}
  Combining \eqref{eq4.11} and \eqref{eq4.13}, we see that
  \begin{equation*}
    h^r_{11}=0,~r=2,\cdots,n
  \end{equation*}
  which implies that $g(h(X,X),JY)=0$ for all orthogonal unit vectors $X,Y$ in $T_pN$. Thus
  \begin{equation}\label{eq4.14}
    h^r_{ii}=0,~1\leq i\not=r\leq n.
  \end{equation}
  In particular, $h^1_{ii}=0$ for $i>1$. This together with \eqref{eq4.9} gives that $h^1_{11}=0$, which implies that $g(h(X,X),JX)=0$ for any unit vector $X$ in $T_pN$. Thus we have
  \begin{equation}\label{eq4.15}
    h^i_{ii}=0,~i=1,\cdots,n.
  \end{equation}
  Also the equality of \eqref{eq4.7} holds, it follows that
  \begin{equation*}
    h^1_{ij}=0, ~2\leq i\not=j\leq n,
  \end{equation*}
  which implies that $g(h(Y,Z),JX)=0$ for all orthogonal unit vectors $X,Y,Z$ in $T_pN$. Thus we have
  \begin{equation}\label{eq4.16}
    h^r_{ij}=0, \textrm{~for~ distinct}~ r,i,j=1,\cdots, n.
  \end{equation}
  From \eqref{eq4.14}, \eqref{eq4.15}, \eqref{eq4.16} and \eqref{symmetryofh}, we know that $p$ is a totally geodesic point.

  For the case $n=2$, we may choose the unit vector $X$ in $T_pN$ such that $JX$ is parallel to the mean curvature vector $H$. Then $\sum\limits_{i=1}^2h^2_{ii}=0$. Combining this with \eqref{eq4.11}, we see that $h^2_{11}=h^2_{22}=0$. Write $h^2_{12}=\mu$, then \eqref{symmetryofh} gives that $h^1_{22}=\mu$, and $h^1_{12}=h^2_{11}=0$. Write $h^1_{11}=\lambda$, then \eqref{eq4.9} implies that $\lambda=3\mu$. Therefore, $p$ is an $H$-umbilical point with $\lambda=3\mu$.

  The converse can be easily verified.

\end{proof}


\section{A lower bound for the Ricci curvature}
\vskip 0.4 true cm

Theorem \ref{thm4.1} shows an upper bound for the Ricci curvature for Lagrangian submanifolds of K$\ddot{\mathrm{a}}$hler QCH-manifolds. In this section we will obtain a lower bound for the Ricci curvature, from which we can generalize Theorem \ref{Theorem1.3}.

\begin{thm}\label{thm5.1}
  Let $N$ be a Lagrangian submanifold of real dimension $n(\geq 3)$ of a K$\ddot{\mathrm{a}}$hler QCH-manifolds $M$. Then for any point $p\in N$ and any unit vector $X\in T_pN$, we have
  \begin{align}\label{eq5.1}
    Ric(X) & \geq -\frac{(n-2)(n-1)(n+1)}{8}a-\frac{n-2}{8}\{n-[\eta(X)^2+\tilde\eta(X)^2]\}b\\\notag
    & -\{(n-1)[\|\eta^{\top}\|^2\|\tilde\eta^{\top}\|^2-g(\eta^{\top},\tilde\eta^{\top})^2]-\|\eta(X)\tilde\eta^{\top}-\tilde\eta(X)\eta^{\top}\|^2\}c\\\notag
    & +(n-1)\tau-\frac{(3n-1)(n-2)n^2}{2(3n+5)}\|H\|^2.\notag
  \end{align}
  The equality holds for any unit tangent vector at $p$ if and only if $p$ is a totally geodesic point.
\end{thm}

\begin{proof}
  Let $X$ be a unit vector in $T_pN$. We choose an orthonormal basis $\{e_1,\cdots,e_n\}$ in $T_pN$ such that $e_1=X$. From the equation \eqref{scalarcuvature2} and \eqref{eqRic}, we have
  \begin{align}\label{eq5.2}
  (n-1)\tau-Ric(X) & =(n-1)\sum_{1\leq i<j\leq n}\bar R(e_i,e_j,e_j,e_i)-\sum_{j=2}^n\bar R(e_1,e_j,e_j,e_1)\\\notag
  & +(n-1)\sum_{r=1}^n\sum_{1\leq i<j\leq n}[h^r_{ii}h^r_{jj}-(h^r_{ij})^2]-\sum_{r=1}^n\sum_{j=2}^n[h^r_{11}h^r_{jj}-(h^r_{1j})^2]\notag.
  \end{align}
  It follows from \eqref{Riemanniancurvature} that
  \begin{align}\label{eq5.3}
    \sum_{1\leq i<j\leq n}\bar R(e_i,e_j,e_j,e_i) & = a\sum_{1\leq i<j\leq n}\pi(e_i,e_j,e_j,e_i)+b\sum_{1\leq i<j\leq n}\Phi(e_i,e_j,e_j,e_i)\\\notag
    & +c\sum_{1\leq i<j\leq n}\Psi(e_i,e_j,e_j,e_i).
  \end{align}
  According to \eqref{QCH-pi}, \eqref{QCH-Phi}, \eqref{QCH-Psi}, and noting that $g(e_i,e_j)=\delta_{ij}$, $g(e_i,Je_j)=0$ for $i,j=1,\cdots,n$, by direct calculations, we have
  \begin{equation}\label{eq5.4}
    \sum_{1\leq i<j\leq n}\pi(e_i,e_j,e_j,e_i)=\frac{(n-1)n}{8},
  \end{equation}
   \begin{align}\label{eq5.5}
    \sum_{1\leq i<j\leq n}\Phi(e_i,e_j,e_j,e_i)& =\frac{1}{8}\sum_{1\leq i\not=j\leq n}g(e_i,e_i)[\eta(e_j)^2+\tilde\eta(e_j)^2]\\\notag
    & =\frac{1}{8}\sum_{1\leq i\not=j\leq n}[\eta(e_j)^2+\eta(Je_j)^2]\\\notag
    & =\frac{n-1}{8},\notag
  \end{align}
  \begin{align}\label{eq5.6}
     \sum_{1\leq i<j\leq n}\Psi(e_i,e_j,e_j,e_i) & =\sum_{1\leq i\not=j\leq n}[\eta(e_j)^2\tilde\eta(e_i)^2-\eta(e_i)\eta(e_j)\tilde\eta(e_i)\tilde\eta(e_j)]\\\notag
    &= \|\eta^{\top}\|^2\|\tilde\eta^{\top}\|^2-g(\eta^{\top},\tilde\eta^{\top})^2.\\\notag
  \end{align}
  By substituting \eqref{eq5.4}, \eqref{eq5.5} and \eqref{eq5.6} into \eqref{eq5.3}, we get
  \begin{equation}\label{eq5.7}
    \sum_{1\leq i<j\leq n}\bar R(e_i,e_j,e_j,e_i) =\frac{(n-1)n}{8}a+\frac{n-1}{8}b+[\|\eta^{\top}\|^2\|\tilde\eta^{\top}\|^2-g(\eta^{\top},\tilde\eta^{\top})^2]c.
  \end{equation}

  On the other hand, by using \eqref{symmetryofh} we calculate that
  \begin{align}\label{eq5.8}
    & (n-1)\sum_{r=1}^n\sum_{1\leq i<j\leq n}[h^r_{ii}h^r_{jj}-(h^r_{ij})^2]-\sum_{r=1}^n\sum_{j=2}^n[h^r_{11}h^r_{jj}-(h^r_{1j})^2]\\\notag
   =& \frac{n-1}{2}\sum_{r=1}^n\sum_{1\leq i\not=j\leq n}h^r_{ii}h^r_{jj}-\sum_{r=1}^n\sum_{j=2}^nh^r_{11}h^r_{jj}-\frac{n-1}{2}\sum_{r=1}^n\sum_{1\leq i\not=j\leq n}(h^r_{ij})^2+\sum_{r=1}^n\sum_{j=2}^n(h^r_{1j})^2\\\notag
   =& \frac{n-1}{2}\sum_{r=1}^n\sum_{1\leq i\not=j\leq n}h^r_{ii}h^r_{jj}-\sum_{r=1}^n\sum_{j=2}^nh^r_{11}h^r_{jj}-(n-1)\sum_{1\leq i\not=r\leq n}(h^i_{rr})^2\\\notag
   - & \frac{n-1}{2}\sum_{\substack{r,i,j~ {\mathrm distinct}\\ 1\leq r,i,j\leq n}}(h^r_{ij})^2+\sum_{j=2}^n(h^j_{11})^2+\sum_{j=2}^n(h^1_{jj})^2+\sum_{2\leq i\not=j\leq n}(h^1_{ij})^2\\\notag
   =& \frac{n-1}{2}\sum_{r=1}^n\sum_{1\leq i\not=j\leq n}h^r_{ii}h^r_{jj}-\sum_{r=1}^n\sum_{j=2}^nh^r_{11}h^r_{jj}-(n-1)\sum_{1\leq i\not=r\leq n}(h^i_{rr})^2+\sum_{j=2}^n(h^j_{11})^2\\\notag
   + & \sum_{j=2}^n(h^1_{jj})^2-\frac{3n-5}{2}\sum_{2\leq i\not=j\leq n}(h^1_{ij})^2- \frac{n-1}{2}\sum_{\substack{r,i,j~ {\mathrm distinct}\\ 2\leq r,i,j\leq n}}(h^r_{ij})^2\\\notag
   \end{align}
   \begin{align*}
   \leq & \frac{n-1}{2}\sum_{r=1}^n\sum_{1\leq i\not=j\leq n}h^r_{ii}h^r_{jj}-\sum_{r=1}^n\sum_{j=2}^nh^r_{11}h^r_{jj}-(n-1)\sum_{1\leq i\not=r\leq n}(h^i_{rr})^2\\\notag
   + & \sum_{j=2}^n(h^j_{11})^2+\sum_{j=2}^n(h^1_{jj})^2.\\\notag
  \end{align*}
  Substituting \eqref{eq4.6}, \eqref{eq5.7} and \eqref{eq5.8} into \eqref{eq5.2}, we get
  \begin{align}\label{eq5.9}
    (n-1)\tau-Ric(X) &\leq \frac{(n-2)(n-1)(n+1)}{8}a+\frac{n-2}{8}\{n-[\eta(X)^2+\tilde\eta(X)^2]\}b\\\notag
    & +\{(n-1)[\|\eta^{\top}\|^2\|\tilde\eta^{\top}\|^2-g(\eta^{\top},\tilde\eta^{\top})^2]-\|\eta(X)\tilde\eta^{\top}-\tilde\eta(X)\eta^{\top}\|^2\}c\\\notag
    & +\frac{n-1}{2}\sum_{r=1}^n\sum_{1\leq i\not=j\leq n}h^r_{ii}h^r_{jj}-\sum_{r=1}^n\sum_{j=2}^nh^r_{11}h^r_{jj}-(n-1)\sum_{1\leq r\not=i\leq n}(h^r_{ii})^2\\\notag
    & +\sum_{r=2}^n(h^r_{11})^2+\sum_{j=2}^n(h^1_{jj})^2.\notag
  \end{align}

Consider the quadratic forms $f_1,f_r:\mathbb{R}^n\to\mathbb{R},~r=2,\cdots,n$ defined respectively by
  \begin{align}\label{eq5.10}
    & f_1(h^1_{11},\cdots,h^1_{nn})\\\notag  =&\frac{n-1}{2}\sum_{1\leq i\not=j\leq n}h^1_{ii}h^1_{jj}-\sum_{j=2}^nh^1_{11}h^1_{jj}-(n-1)\sum_{i=2}^n(h^1_{ii})^2+\sum_{j=2}^n(h^1_{jj})^2\\\notag
   =&-(n-2)\sum_{i=2}^n(h^1_{ii})^2+(n-2)\sum_{j=2}^nh^1_{11}h^1_{jj}+(n-1)\sum_{2\leq i<j\leq n}h^1_{ii}h^1_{jj},\notag
  \end{align}
   \begin{align}\label{eq5.11}
    & f_r(h^r_{11},\cdots,h^r_{nn})\\\notag  =&\frac{n-1}{2}\sum_{1\leq i\not=j\leq n}h^r_{ii}h^r_{jj}-\sum_{j=2}^nh^r_{11}h^r_{jj}-(n-1)\sum_{1\leq i(\not=r)\leq n}(h^r_{ii})^2+(h^r_{11})^2\\\notag
     =&(n-2)\sum_{j=2}^nh^r_{11}h^r_{jj}+(n-1)\sum_{2\leq i<j\leq n}h^r_{ii}h^r_{jj}-(n-2)(h^r_{11})^2-(n-1)\sum_{1\leq i(\not=1,r)\leq n}(h^r_{ii})^2.\\\notag
  \end{align}
  From Example \ref{ex3.4}, we know that
  \begin{equation}\label{eq5.12}
    f_1(h^1_{11},\cdots,h^1_{nn})\leq\frac{(n-2)(n-1)}{2(n+1)}(\sum_{i=1}^nh^1_{ii})^2,
  \end{equation}
  with the equality holding if and only if
  \begin{equation}\label{eq5.13}
   h^1_{22}=\cdots=h^1_{nn}=\frac{1}{2}h^1_{11}.
  \end{equation}
  From Example \ref{ex3.5}, we know that
  \begin{equation}\label{eq5.14}
    f_r(h^r_{11},\cdots,h^r_{nn})\leq\frac{(n-2)(3n-1)}{2(3n+5)}(\sum_{i=1}^nh^r_{ii})^2,~r=2,\cdots,n
  \end{equation}
  with the equality holding if and only if
  \begin{equation}\label{eq5.15}
   h^r_{22}=\cdots=h^r_{r-1r-1}=h^r_{r+1r+1}=\cdots=h^r_{nn}=\frac{3}{2}h^r_{11},~ h^r_{rr}=\frac{9}{2}h^r_{11}.
  \end{equation}
  From \eqref{eq5.9}, \eqref{eq5.10}, \eqref{eq5.11},  \eqref{eq5.12}and \eqref{eq5.14}, we obtain that
  \begin{align}\label{eq5.16}
    (n-1)\tau-Ric(X)\leq & \frac{(n-2)(n-1)(n+1)}{8}a+\frac{n-2}{8}\{n-[\eta(X)^2+\tilde\eta(X)^2]\}b\\\notag
     +& \{(n-1)[\|\eta^{\top}\|^2\|\tilde\eta^{\top}\|^2-g(\eta^{\top},\tilde\eta^{\top})^2]-\|\eta(X)\tilde\eta^{\top}-\tilde\eta(X)\eta^{\top}\|^2\}c\\\notag
     +& \frac{(n-2)(n-1)}{2(n+1)}(\sum_{i=1}^nh^1_{ii})^2+\frac{(n-2)(3n-1)}{2(3n+5)}\sum_{r=2}^n(\sum_{i=1}^nh^r_{ii})^2.\\\notag
  \end{align}
  Noting that $n\geq 3$, $\frac{(n-2)(n-1)}{n+1}<\frac{(n-2)(3n-1)}{3n+5}$, \eqref{eq5.16} gives \eqref{eq5.1} immediately.

   Now we consider the equality case. Suppose the equality of \eqref{eq5.1} holds for any unit vector $X$ in $T_pN$. It follows from \eqref{eq5.16} and \eqref{eq5.1} that
  \begin{equation}\label{eq5.17}
    \sum_{i=1}^nh^1_{ii}=0.
  \end{equation}
  Combining this with \eqref{eq5.13}, we see that
  \begin{equation*}
    h^1_{11}=\cdots=h^1_{nn}=0,
  \end{equation*}
  wihich implies that $g(h(X,X),JX)=0$ for any unit vector $X\in T_pN$, and $g(h(Y,Y),JX)=0$ for all orthogonal unit vectors $X,Y$ in $T_pN$. Thus
  \begin{equation}\label{eq5.18}
    h^r_{ii}=0,~r,i=1,\cdots,n.
  \end{equation}
  Also the equality of \eqref{eq5.8} holds, it gives that
  \begin{equation}\label{eq5.19}
    h^1_{ij}=0,~2\leq i\not=j\leq n,
  \end{equation}
  \begin{equation}\label{eq5.20}
    h^r_{ij}=0,~\textrm{for~distinct}~r,i,j=2,\cdots,n.
  \end{equation}
  Combining \eqref{eq5.18} and \eqref{symmetryofh} we see that
  \begin{equation}\label{eq5.21}
    h^1_{1i}=0,~i=1,\cdots,n,
  \end{equation}
  \begin{equation}\label{eq5.22}
    h^r_{rj}=0,~ j=1,\cdots,n,~ r=2,\cdots,n.
  \end{equation}
  From \eqref{eq5.18}-\eqref{eq5.22} we see that $h^r_{ij}=0$ for $i,j,r=1,\cdots,n$, which implies thah $p$ is a totally geodesic point.

  The converse can be easily verified.
\end{proof}

Theorem \ref{thm5.1} immediately gives an inequality related to T.Oprea's invariant $\delta_n(M)$ defined by \eqref{01.1} as follows.

\begin{thm}\label{Theorem5.2}
  Let $N$ be a Lagrangian submanifold of real dimension $n(\geq 3)$ of a K$\ddot{\mathrm{a}}$hler QCH-manifolds $M$. Then
  \begin{align*}
    \delta_n(M) & \leq \frac{(n-2)(n+1)}{8}a+\frac{n-2}{8(n-1)}\{n-[\eta(X)^2+\tilde\eta(X)^2]\}b\\
    & +\big[\|\eta^{\top}\|^2\|\tilde\eta^{\top}\|^2-g(\eta^{\top},\tilde\eta^{\top})^2-\frac{1}{n-1}\|\eta(X)\tilde\eta^{\top}+\tilde\eta(X)\eta^{\top}\|^2\big]c\\
    & +\frac{(3n-1)(n-2)n^2}{2(n-1)(3n+5)}\|H\|^2.
  \end{align*}
\end{thm}

\begin{rmk}
  For $b=0,c=0$, Theorem \ref{Theorem5.2} is due to Theorem \ref{Theorem1.3}.
\end{rmk}

%

\vskip 0.5 true cm


\bigskip
\bigskip

\noindent {\footnotesize {\it L. Zhang} \\
{School of Mathematics and Computer Science, Anhui Normal University}\\
{Anhui 241000, P.R. China}\\
{Email: zhliang43@ahnu.edu.cn}

\vskip 0.5 true cm

\noindent {\footnotesize {\it X. Liu} \\
{School of Mathematics and Computer Science, Anhui Normal University}\\
{Anhui 241000, P.R. China}\\
{Email: 2287354429@qq.com}

\vskip 0.5 true cm

\noindent {\footnotesize {\it D. Cai} \\
{School of Mathematics and Computer Science, Anhui Normal University}\\
{Anhui 230026, P.R. China}\\
{Email: 2561408739@qq.com}

\end{document}